\documentclass[11pt,reqno]{amsart}
\usepackage[T1]{fontenc}
\usepackage{lmodern}
\usepackage{amsmath,amssymb,mathtools}
\usepackage{microtype}
\usepackage[a4paper,margin=26mm]{geometry}
\usepackage[hidelinks]{hyperref}
\hypersetup{pdftitle={Non-compact convex hulls in CAT(0) spaces revisited},pdfsubject={A three-point refinement of the graph-and-cone construction},pdfauthor={Alexander Lytchak}}
\newtheorem{theorem}{Theorem}[section]
\newtheorem{proposition}[theorem]{Proposition}
\newtheorem{lemma}[theorem]{Lemma}

\theoremstyle{remark}

\newcommand{\conv}{\operatorname{conv}}
\newcommand{\cconv}{\overline{\operatorname{conv}}}
\newcommand{\pconv}{\operatorname{conv}_{\pi}}

\newcommand{\clB}{\overline B}
\numberwithin{equation}{section}
\allowdisplaybreaks[1]
\title[Non-compact convex hulls revisited]{Non-compact convex hulls in CAT(0) spaces revisited}
\author[]{Alexander Lytchak}
\date{}
\subjclass[2020]{53C23, 51F99, 52A05}
\keywords{Hadamard space, convex hull, Euclidean cone}

\begin{document}
\begin{abstract}
We simplify Goodwin's construction of a non-compact closed convex hull in a
CAT(0) space and show that three points suffice for the example.
\end{abstract}
\maketitle

\section{Introduction}
In \cite{Goo26}, Ariel Goodwin answered a basic question in the theory of
non-positive curvature: must the closed convex hull of a compact subset of a
complete CAT(0) space be compact? The question had been open since
Gromov's foundational work \cite[6.B$_1$(f)]{Gro93}. It was restated, for
example, in \cite[Section~9.O]{AKP24}, \cite[Section~3.1]{Bac23}, and
\cite[Section~12]{Pet25}. Indications that the answer might be negative
appeared in the work of Monod \cite{Mon16}, Lytchak--Petrunin \cite{LP22},
and Basso--Krifka--Soultanis \cite{BKS24}.

The presentation in \cite{Goo26} relies on a nontrivial existence theorem
for regular graphs of large girth. In this note we give a simplified
version of the argument, starting from scratch.

The main result is the following minor refinement of \cite{Goo26}.

\begin{theorem}\label{thm:main}
There exist a proper, unbounded metric graph $Y$ of girth at least $2\pi$
and a three-point subset $A\subset Y$ with the following property.

Under the canonical identification of $Y$ with the unit sphere about
the tip $o$ of the Euclidean cone $X=C(Y)$, the convex hull of $A$ in the
CAT(0) space $X$ contains a neighborhood of $o$. In particular, the closure
of the convex hull of $A$ is not compact.
\end{theorem}

A few comments are in order.
\begin{itemize}
\item The \emph{girth} of a metric graph is the infimum of the lengths of
its simple closed curves. The condition that the girth be at least $2\pi$
is equivalent to the CAT(1) condition on $Y$, which is in turn equivalent
to the CAT(0) condition on its Euclidean cone; see
\cite[Chapter~II.5 and Theorem~II.3.14]{BH99}.

\item \emph{Properness} in Theorem~\ref{thm:main} means that all closed
bounded balls are compact. In our construction, every closed bounded ball
meets only finitely many edges.

\item In the constructed graph, all vertex degrees lie between two and six.

\item The Euclidean cone $X$ is complete and geodesically complete;
compare \cite[Section~4.2]{LN19}. It is locally compact away from $o$,
whereas no neighborhood of $o$ is compact.

\item The CAT(0) space $X$ has Hausdorff and topological dimension two.
Away from the tip $o$, it is locally bilipschitz to the product of a
locally finite metric graph and the real line.

\item Already three points suffice to produce a non-compact closed
convex hull. The three-point case was explicitly mentioned in
\cite[Section~9.O]{AKP24}, \cite[Section~1]{BKS24},
\cite[Section~12]{Pet25}, and \cite{Goo26}.

\item The unbounded CAT(1) graph $Y$ has the counterintuitive property
that it is the $\pi$-convex hull of a three-point subset $A$. Thus the
convex hull of $A$ in $X$ contains an infinitesimal neighborhood of the
tip $o$. This is reminiscent of the main result of \cite{LP22}.

\item To ensure that the convex hull of $A$ contains an actual
neighborhood of $o$, rather than only an infinitesimal one, we need
precise control of the lengths of the edges that generate the convex
hull $Y$ step by step. This central idea of \cite{Goo26} is reminiscent
of \cite{Mon16}.
\end{itemize}

\subsection*{Acknowledgements}
The author would like to thank Anton Petrunin for many conversations about
this topic over the years.
The author was supported in part by the German Research Foundation under
Germany's Excellence Strategy -- EXC-2047/1 -- 390685813.
This paper arose in close collaboration with ChatGPT, an AI assistant
developed by OpenAI. The collaboration involved exploring proof strategies,
refining arguments, locating references, and drafting and revising the
manuscript. The author takes full responsibility for the mathematical
content and the final text.

\section{Convex hulls}\label{sec:hulls}

\subsection{Notation}
In a metric space, we denote by $\clB_r(x)$ the closed ball of radius $r$
centered at $x$.
For $s>0$, a subset is \emph{$s$-separated} if its distinct points have
pairwise distances at least~$s$.

A graph will mean a connected, locally finite graph with edges of positive
length, endowed with the associated path metric. Vertices of
degree two are allowed. All graphs in the construction have finitely many
edges, except for the final increasing union, whose properness is proved
below. All CAT(0) and CAT(1) spaces are complete by convention; a CAT(0)
space is also called a Hadamard space.

All graphs considered below are geodesic spaces, by finiteness at the
finite stages and properness at the final step. In a graph $G$ of girth at least $2\pi$,
points $x,y$ with $d_G(x,y)<\pi$ are joined by a unique geodesic,
denoted by $xy$; we set $xx=\{x\}$.

A subset $H\subset G$ is \emph{$\pi$-convex} if it contains $xy$ whenever
$x,y\in H$ and $d_G(x,y)<\pi$. Equivalently, these points can be joined
inside $H$ by a curve realizing their distance in $G$. The
\emph{$\pi$-convex hull} $\pconv^G(A)$ is the smallest $\pi$-convex subset
containing $A$. It is obtained by iteration:
\begin{equation}\label{eq:pi-hull-iteration}
 A_1=A,\qquad
 A_{i+1}=A_i\cup
 \bigcup_{\substack{x,y\in A_i\\ d_G(x,y)<\pi}}xy,
 \qquad
 \pconv^G(A)=\bigcup_{i=1}^{\infty}A_i.
\end{equation}
We use the same definition and notation in any CAT(1) space.

In a CAT(0) space $X$, all pairs of points are joined by unique geodesics.
The convex hull is given by the analogous construction without a length
restriction:
\begin{equation}\label{eq:hull-iteration}
 A_1=A,\qquad A_{i+1}=A_i\cup\bigcup_{x,y\in A_i}xy,
 \qquad \conv^X(A)=\bigcup_{i=1}^{\infty}A_i.
\end{equation}
The closure of $\conv^X(A)$ is the \emph{closed convex hull} of $A$,
denoted by $\cconv^X(A)$.

\subsection{The graph to be constructed}
The following statement isolates the graph construction and the length
control needed for a uniform neighborhood in the cone.

\begin{proposition}\label{prop:graph}
There exist a proper, unbounded metric graph $Y$ of girth at least $2\pi$,
a three-point set $A\subset Y$, and a sequence $(\beta_i)_{i\geq1}$ such that
\begin{equation}\label{eq:beta-summability}
 0<\beta_i<\frac{\pi}{2},\qquad
 \sum_{i=1}^{\infty}\beta_i^2<\infty.
\end{equation}
The sets defined by
\begin{equation}\label{eq:controlled-hulls}
 A_1=A,\qquad
 A_{i+1}=A_i\cup
 \bigcup_{\substack{x,y\in A_i\\ d_Y(x,y)\leq2\beta_i}}xy
\end{equation}
exhaust the graph:
\begin{equation}\label{eq:controlled-exhaustion}
 Y=\bigcup_{i=1}^{\infty}A_i.
\end{equation}
In particular, $\pconv^Y(A)=Y$.
\end{proposition}

The proof is given in Section~\ref{sec:construction}. Here we record an
elementary lemma.

\begin{lemma}\label{lem:cos-product}
If $(\beta_i)_{i\geq1}$ satisfies \eqref{eq:beta-summability}, then
\begin{equation}\label{eq:cos-product}
 r_*:=\prod_{i=1}^{\infty}\cos\beta_i>0.
\end{equation}
\end{lemma}

\begin{proof}
The infinite product converges to a positive number if
\[
 \sum_{i=1}^{\infty}|\log(\cos\beta_i)|<\infty.
\]
Since $\beta_i\to0$, Taylor expansion at zero gives
\[
 0<-\log(\cos\beta_i)<2\beta_i^2
\]
for all sufficiently large $i$. The assumed summability of
$\sum_i\beta_i^2$ proves the claim.
\end{proof}

\subsection{Passage to the Euclidean cone}\label{sec:cone}
In the Euclidean cone $X=CY$, we denote the tip by $o$ and write $(r,y)$
for the point at distance $r>0$ from $o$ in direction $y\in Y$; see
\cite[Chapter~I.5]{BH99}. If $Y$ is CAT(1), then $X$ is CAT(0) by
\cite[Theorem~II.3.14]{BH99}.

\begin{proposition}\label{prop:ball}
Let $Y$ be an unbounded CAT(1) space, and let $A\subset Y$, the sequence
$(\beta_i)$, and the sets $A_i\subset Y$ satisfy
\eqref{eq:beta-summability}--\eqref{eq:controlled-exhaustion}.
In $X=CY$, the set
$\widehat A=\{(1,a):a\in A\}$ satisfies
\begin{equation}\label{eq:ball}
 \clB_{r_*}(o)\subset\conv^X(\widehat A)\subset\clB_1(o),
\end{equation}
with $r_*>0$ as in \eqref{eq:cos-product}.

In particular, $\cconv^X(\widehat A)$ is not compact.
\end{proposition}

\begin{proof}
Set $C=\conv^X(\widehat A)$ and consider the set of directions represented
by $C$:
\[
 D=\{y\in Y:(r,y)\in C\text{ for some }r>0\}.
\]
This set is $\pi$-convex. Indeed, for distinct $y,z\in D$ with
$d_Y(y,z)<\pi$, choose positive-radius points of $C$ in these directions.
Their geodesic is a chord in the Euclidean sector over $yz$; it avoids
$o$, and its radial projection to $Y$ has image $yz$. Thus $yz\subset D$.
Since $A\subset D$ and $\pconv^Y(A)=Y$, we infer $D=Y$.

As $Y$ is unbounded, choose $y,z\in Y$ with $d_Y(y,z)\geq\pi$.
By $D=Y$, there exist $r,s>0$ such that $(r,y),(s,z)\in C$.
The geodesic between these two cone points passes through $o$, so $o\in C$.
Consequently, radial contraction towards $o$ preserves $C$.

Define the successive radii by
\begin{equation}\label{eq:radii}
 r_1=1,\qquad r_{i+1}=r_i\cos\beta_i.
\end{equation}
We claim that $(r_i,x)\in C$ for every $x\in A_i$. The case $i=1$
holds because $r_1=1$ and $\{(1,a):a\in A_1\}=\widehat A\subset C$.

For the induction step, let $x\in A_{i+1}$. There exist $y,z\in A_i$
with $0\leq\theta=d_Y(y,z)\leq2\beta_i<\pi$ and $x\in yz$;
when $x\in A_i$, we may take $y=z=x$.
By induction, $(r_i,y),(r_i,z)\in C$. In the possibly degenerate Euclidean
sector over $yz$,
the chord between these points meets every direction of $yz$ and has
minimum distance
\[
 r_i\cos(\theta/2)\geq r_i\cos\beta_i=r_{i+1}
\]
from $o$. Thus $(t,x)\in C$ for some $t\geq r_{i+1}$. Contracting
radially yields $(r_{i+1},x)\in C$, completing the induction.

By Lemma~\ref{lem:cos-product}, the radii $r_i$ decrease to $r_*>0$.
For every $x\in Y$, the exhaustion provides an index $i$ with $x\in A_i$.
Since $(r_i,x)\in C$ and $o\in C$, the convex hull $C$ contains every
point $(t,x)$
with $0<t\leq r_*$. This proves the first inclusion in \eqref{eq:ball};
the second follows from convexity of closed balls in a CAT(0) space.

Finally, unboundedness of $Y$ supplies points $y_j$ with
$d_Y(y_i,y_j)\geq\pi$ for $i\neq j$.
Then the points $(r_*,y_j)\in C$ have pairwise distance exactly
$2r_*$. Hence $C$ is not totally bounded, so its closure cannot be compact.
\end{proof}

\begin{proof}[Proof of Theorem~\ref{thm:main}, assuming Proposition~\ref{prop:graph}]
A proper metric graph of girth at least $2\pi$ is CAT(1), cf.
\cite[Proposition~6.8(2)]{Bal04}.
Thus the graph $Y$ in Proposition~\ref{prop:graph} is CAT(1), and its
Euclidean cone $X=CY$ is CAT(0). Proposition~\ref{prop:ball} gives the asserted
neighborhood inclusion and non-compactness.
\end{proof}

\section{Construction of the graph}\label{sec:construction}

\subsection{Adding a matching}
We begin with an elementary observation.

\begin{lemma}\label{lem:matching}
Let $G$ be a finite graph of girth at least $2\pi$, and let $W\subset G$
be a finite $s$-separated set of even cardinality. Pair the points of $W$,
and attach an interval of length $0<\ell<\pi$ between the points of each
pair, with interiors disjoint from $G$ and from one another. Suppose that
\[
 s+\ell\geq\pi,\qquad
 d_G(u,v)\geq2\pi-\ell
 \quad\text{for every paired pair }u,v\in W.
\]
Then the enlarged graph has girth at least $2\pi$. The midpoints of the
new intervals form an $(s+\ell)$-separated set and have degree two.
\end{lemma}

\begin{proof}
Consider a simple cycle containing new intervals. If it contains exactly
one new interval, its length is at least $\ell+(2\pi-\ell)=2\pi$.
If it contains at least two such intervals, the matching condition forces
it to contain at least two paths in $G$ between distinct points of $W$.
Hence the cycle has length at least $2\ell+2s\geq2\pi$.

A minimizing path between distinct new midpoints traverses their two
half-intervals and at least one old path between distinct points of $W$.
Its length is therefore at least $\ell/2+s+\ell/2=s+\ell$.
The assertion about degrees follows from the construction.
\end{proof}

\subsection{A finite graph generated by three points}
Using Lemma~\ref{lem:matching}, we can find CAT(1) graphs in which the
$\pi$-convex hull of three points contains arbitrarily large separated sets.

\begin{proposition}\label{prop:seed}
For every integer $N\geq1$, there exist a finite graph $H$, a three-point
set $A\subset H$, and a set $V\subset H$ of exactly $N$ vertices such that
\begin{enumerate}
\item $H$ has girth at least $2\pi$;
\item every vertex $v\in V$ has degree two, and
$d_H(v,w)>2\pi$ for distinct $v,w\in V$;
\item for some $0<\ell<\pi$, the sequence defined by
\begin{equation}\label{eq:seed-generation}
 A_1=A,\qquad
 A_{i+1}=\bigcup_{\substack{x,y\in A_i\\d_H(x,y)\leq\ell}}xy
 \quad(i\geq1)
\end{equation}
satisfies $A_4=H$.
\end{enumerate}
\end{proposition}

\begin{proof}
Consider a circle $G=H_0$ of length $L\in(2\pi,3\pi)$. Choose two points
$x^\pm\in G$ with $d_G(x^+,x^-)>\pi$, and choose $2N$ distinct points
$x_i^+$ near $x^+$ and $2N$ distinct points $x_i^-$ near $x^-$, sufficiently
close that $d_G(x_i^+,x_j^-)>\pi$ for all $i,j$.
Let $W=\{x_i^\pm:1\leq i\leq2N\}$, and let $s>0$ be the minimum
distance between its distinct points. Choose $t$ such that
\[
 0<t<\min\{s,\pi\},\qquad d_G(x_i^+,x_j^-)>\pi+t
 \quad\text{for all }i,j.
\]
Set $\ell=\pi-t/2$, choosing $t$ small enough that $L/3<\ell$.

\emph{First matching.} Pair $x_i^+$ with $x_i^-$. Apply
Lemma~\ref{lem:matching} with separation parameter $s$: its assumptions
hold since $s+\ell>\pi$ and
$d_G(x_i^+,x_i^-)>\pi+t>2\pi-\ell$.
Attach intervals of length $\ell$ to obtain $H_1$, and denote their
midpoints by $u_i$. The set $W_1=\{u_i:1\leq i\leq2N\}$ is
$(s+\ell)$-separated, and
\[
 s+\ell>t+\ell=2\pi-\ell.
\]

\emph{Second matching.} Pair the points of $W_1$ arbitrarily and again
attach intervals of length $\ell$. Lemma~\ref{lem:matching}, with separation
parameter $s+\ell$, gives a graph $H=H_2$ of girth at least $2\pi$
and a set $V$ of exactly $N$ new midpoints satisfying
\[
 d_H(v,w)\geq s+2\ell=2\pi+(s-t)>2\pi
 \qquad(v,w\in V,\ v\neq w).
\]
Every attachment point is used in only one matching. Thus all vertex
degrees lie between two and three, and the vertices of $V$ have degree two.

Finally, choose $A$ to consist of three equally spaced points on the
original circle $G$. The three
circular arcs between consecutive points of $A$ have length $L/3<\ell$,
and every interval attached later has length $\ell<\pi$. By the girth bound,
all these paths remain unique geodesics in $H$. Thus the sequence in
\eqref{eq:seed-generation}, computed in $H$, satisfies
\[
 H_0\subset A_2,\qquad H_1\subset A_3,\qquad H=H_2\subset A_4.
\]
Since every $A_i$ is contained in $H$, this proves $A_4=H$ and hence
$\pconv^H(A)=H$.
\end{proof}

\subsection{The infinite sequence of attachments}\label{sec:link}
For $n\geq0$, set
\begin{equation}\label{eq:alpha}
 \alpha_n=(n+1)^{-2/3},\qquad
 s_0=0,\qquad s_n=\sum_{j=0}^{n-1}\alpha_j\quad(n\geq1).
\end{equation}
Then $0<\alpha_n\leq1$, $\alpha_n$ decreases to zero, and
\begin{equation}\label{eq:sums}
 \sum_{n=0}^{\infty}\alpha_n=\infty,\qquad
 \sum_{n=0}^{\infty}\alpha_n^2<\infty.
\end{equation}
Choose an integer $N\geq1$ and use the construction of
Proposition~\ref{prop:seed} to obtain
\[
 G_0:=H,\qquad V_0:=V.
\]
In particular, $|V_0|=N$, every vertex of $V_0$ has degree two, and distinct
points of $V_0$ have distance greater than $2\pi$ in $G_0$.

We will construct a sequence of finite graphs
$G_0\subset G_1\subset\cdots$, nonnegative integers $k_n$, and subsets
$V_n\subset G_n$ with the following properties:
\begin{itemize}
\item $k_n=|V_n|$, so $k_0=N$;
\item each $V_n$ consists of vertices of degree two in $G_n$;
\item $G_{n+1}$ is obtained from $G_n$ by attaching exactly $k_{n+1}$
intervals of length $2\alpha_n$ between pairs of distinct points of $V_n$,
with new, mutually disjoint interiors. These intervals are subdivided at
their midpoints, which form $V_{n+1}$;
\item each $G_n$ has girth at least $2\pi$;
\item all vertex degrees in $G_n$ lie between two and six.
\end{itemize}
We specify the attachments in the next subsection. First, we record the
consequences of these properties, without assuming that every $k_n$ is
positive.

If $k_{n+1}=0$, then $G_{n+1}=G_n$ and $V_{n+1}=\varnothing$, so all
subsequent graphs equal $G_n$. Thus we allow the sequence to stabilize.
In either case, put
\[
 G_\infty=\bigcup_{n=0}^{\infty}G_n.
\]
If the process terminates after finitely many stages, this union is just
the final finite graph.

The inclusion $G_n\subset G_{n+1}$ preserves the assigned lengths of
edges, but need not preserve distances: new intervals may create
shortcuts. Denoting the path metric of $G_n$ by $d_n$, we have
$d_{n+1}\leq d_n$ on $G_n$, and the inequality may be strict.

On the other hand, in a graph of girth at least $2\pi$, every simple path
of length at most $\pi$ is minimizing. Consequently, for all later stages,
\begin{equation}\label{eq:short-stability}
 d_j(x,y)=d_i(x,y)\qquad
 (j\geq i,\ x,y\in G_i,\ d_i(x,y)\leq\pi).
\end{equation}

For $x,y\in G_\infty$, define
\begin{equation}\label{eq:limit-metric}
 d_\infty(x,y)=\inf_{n:\,x,y\in G_n}d_n(x,y).
\end{equation}
Monotonicity of the metrics implies that $d_\infty$ is a pseudometric.
It is nondegenerate: for $x\neq y$, either some $d_i(x,y)\leq\pi$, in
which case all subsequent distances equal this positive value, or all
finite-stage distances exceed $\pi$, so $d_\infty(x,y)\geq\pi$.
In particular, \eqref{eq:short-stability} gives
\begin{equation}\label{eq:metric}
 d_\infty(x,y)=d_i(x,y)\qquad
 (x,y\in G_i,\ d_i(x,y)\leq\pi).
\end{equation}

\begin{proposition}\label{prop:proper}
The metric space $(G_\infty,d_\infty)$ is a proper CAT(1) metric graph
of girth at least $2\pi$. Every closed bounded ball meets only finitely
many edges, and
\[
 \pconv^{G_\infty}(A)=G_\infty.
\]
Either the sequence $(G_n)$ stabilizes and $G_\infty$ is finite, or
$k_n>0$ for every $n$ and $G_\infty$ is unbounded.
\end{proposition}

\begin{proof}
For each $n$, the inclusion $(G_n,d_n)\to(G_\infty,d_\infty)$ is
$1$-Lipschitz and preserves lengths of curves by \eqref{eq:metric}.
Since $G_n$ is compact, this inclusion is also a topological embedding.
Given $x,y\in G_\infty$ and $\eta>0$, choose a stage $m$ containing
both points with $d_m(x,y)<d_\infty(x,y)+\eta$. A $d_m$-geodesic
between them has the same length with respect to $d_\infty$. Thus
$d_\infty$ is a length metric.

Let $m\geq n+2$, $x\in G_n$, and $y\in G_m\setminus G_{n+1}$.
Every path in $G_m$ from $x$ to $y$ must reach $V_{n+1}$ and therefore
traverse a half of one of the $k_{n+1}$ intervals attached in the step
$G_n\to G_{n+1}$. Such a half-interval joins a point of $V_n$ to a point
of $V_{n+1}$ and has length $\alpha_n$. Hence $d_m(x,y)\geq\alpha_n$.
The same argument applies in every later stage containing both points,
so $d_\infty(x,y)\geq\alpha_n$. In~particular,
\[
 \clB_{\alpha_n/2}(x)\subset G_{n+1}\qquad(x\in G_n),
\]
where the ball is taken with respect to $d_\infty$.

More generally, for $k\geq1$, a path from $x\in G_n$ to a point outside
$G_{n+k}$ must successively reach $V_{n+1},\ldots,V_{n+k}$. Before first
reaching $V_{j+1}$, it must traverse a half-interval from $V_j$, of length
$\alpha_j$. These portions of the path are disjoint. Applying this
observation in each finite stage containing the endpoints gives
\[
 d_\infty(x,y)\geq\sum_{j=n}^{n+k-1}\alpha_j
 \qquad(x\in G_n,\ y\in G_\infty\setminus G_{n+k}).
\]
Since $\sum_j\alpha_j=\infty$, for every $x\in G_n$ and $R>0$ we can
choose $k$ so that this sum exceeds $R$. Then
$\clB_R(x)\subset G_{n+k}$. The finite graph $G_{n+k}$ is compact also
for $d_\infty$, by continuity of its inclusion. The ball is closed in
this compact set, hence compact. Thus $G_\infty$ is proper.

Every point has a neighborhood contained in a finite stage. Since the
inclusion of that stage is a topological embedding and is locally isometric
by \eqref{eq:metric}, $G_\infty$ is a metric graph. Every bounded ball
meets only finitely many edges: it is contained in a finite stage, and
each vertex of that stage is incident to at most six edges in $G_\infty$.

Every simple closed curve is compact and therefore contained in a finite
stage. Its length is unchanged by the inclusion, so the girth of
$G_\infty$ is at least $2\pi$. Being a proper metric graph with this girth
bound, $G_\infty$ is CAT(1) by \cite[Proposition~6.8(2)]{Bal04}.

If some $k_n$ vanishes, the sequence stabilizes as observed above.
Otherwise, $G_\infty\setminus G_k$ is nonempty for every $k$. The distance
bound above, with $n=0$ and $k\to\infty$, then shows that $G_\infty$
is unbounded.

Finally, every interval attached in the step $G_n\to G_{n+1}$ remains
a geodesic of length $2\alpha_n<\pi$ in $G_\infty$, by \eqref{eq:metric}.
Consequently,
\begin{equation}\label{eq:stage-hull}
 G_{n+1}\subset G_n\cup
 \bigcup_{\substack{x,y\in G_n\\d_\infty(x,y)\leq2\alpha_n}}xy.
\end{equation}
Every geodesic of length less than $\pi$ in $G_0$ also remains a geodesic
in $G_\infty$. Thus $\pconv^{G_0}(A)=G_0$ implies
$G_0\subset\pconv^{G_\infty}(A)$. Iterating \eqref{eq:stage-hull} gives
\[
 G_\infty\subset\pconv^{G_\infty}(G_0)
 \subset\pconv^{G_\infty}(A)\subset G_\infty,
\]
which proves the required equality.
\end{proof}

\subsection{Attachment rules}
We now specify the attachments. The following procedure produces a sequence
with the properties listed above for every $N\geq1$. The only remaining
issue will be to choose $N$ so that the sequence does not stabilize.

Suppose that $G_n$ and the distinguished set $V_n$ of degree-two vertices
have been constructed. We start with $G_n^0:=G_n$. Once the enlargement $G_n^j$ of $G_n^0$ has been
constructed,
we regard $V_n$ as a subset of $G_n^j$. Distances between its points may
now be much smaller than in $G_n$.

If there are distinct $u,v\in V_n$, both of degree less than six in
$G_n^j$, such that
\[
 d_{G_n^j}(u,v)\geq2\pi-2\alpha_n,
\]
choose such a pair and attach to this pair an interval of length $2\alpha_n$, to obtain $G_n^{j+1}$.
If no such pair exists, stop. Subdivide all new intervals at their
midpoints, put $G_{n+1}:=G_n^j$, and let $V_{n+1}$ consist of these
midpoints. Then $k_{n+1}=|V_{n+1}|=j$.

We observe the following properties:
\begin{enumerate}
\item Each attachment increases the degrees of two vertices of $V_n$ by
one. Each vertex of $V_n$ starts with degree two and can receive at most
four new intervals. Hence the procedure stops after at most $2|V_n|=2k_n$
attachments.
\item Induction on $j$ and Lemma~\ref{lem:matching}, applied to the single
pair $\{u,v\}$ with $\ell=2\alpha_n$ and $s=2\pi-2\alpha_n$, show
that every $G_n^j$ has girth at least $2\pi$. Thus so does $G_{n+1}$.
\item Every vertex of $V_{n+1}$ has degree two in $G_{n+1}$. If all
vertex degrees in $G_n$ lie between two and six, the same holds in
$G_{n+1}$.
\item Caution: We cannot apply Lemma~\ref{lem:matching} to the simultaneous
attachment of all the new intervals, since we deliberately allow several
new intervals to be attached at the same~vertex. Otherwise,
$|V_{n+1}|$ would be smaller than $|V_n|$ whenever $V_n\neq\varnothing$,
and the large-step process $G_0\to G_1\to\cdots$ would eventually terminate.
\end{enumerate}

\begin{lemma}\label{lem:growth}
If the initial cardinality $N$ is sufficiently large, every generation $V_n$
is nonempty.
\end{lemma}

\begin{proof}
After constructing $G_{n+1}$ from $G_n$, denote by $U_n\subset V_n$
the set of vertices at which fewer than four new intervals have been
attached. By the stopping rule, any two points of $U_n$ have distance
less than $2\pi$ in $G_{n+1}$; otherwise, another attachment would be possible.

A minimizing path in $G_{n+1}$ between two such points uses no edge of
$G_0$. Otherwise, it would contain a nonconstant subpath in $G_0$ whose
endpoints belong to $V_0$. These endpoints must be distinct, since a
minimizing path cannot return to a previously visited point. Their distance
in $G_0$ is greater than $2\pi$, a contradiction.

Every elementary edge added after $G_0$ and present in $G_{n+1}$ has length
at least $\alpha_n$, since its length is one of
$\alpha_0,\ldots,\alpha_n$. If $U_n$ is nonempty, fix $u\in U_n$.
Every point of $U_n$ can be reached from $u$ using at most $2\pi/\alpha_n$
new elementary edges. There are at most six choices at each step.
Consequently, for a constant $K>0$ independent of $N$ and $G_0$, we have
\begin{equation}\label{eq:unsaturated}
 |U_n|\leq6^{1+2\pi/\alpha_n}
 =6^{1+2\pi(n+1)^{2/3}}
 \leq K\left(\frac32\right)^n.
\end{equation}
 Certainly, \eqref{eq:unsaturated}  also holds
when $U_n$ is empty.

Choose $N\geq4K$. We prove by induction that
\begin{equation}\label{eq:growth}
 k_n\geq N\left(\frac32\right)^n\qquad(n\geq0).
\end{equation}
The case $n=0$ is immediate. Under the induction hypothesis,
\eqref{eq:unsaturated} gives $|U_n|\leq k_n/4$.

At each vertex in $V_n\setminus U_n$, exactly four new intervals have
been attached. Each interval contributes one midpoint to $V_{n+1}$,
and each midpoint is counted at most twice in this incidence count. Hence
\[
 k_{n+1}=|V_{n+1}|\geq2|V_n\setminus U_n|
 \geq\frac32k_n\geq N\left(\frac32\right)^{n+1}.\qedhere
\]
\end{proof}

We can now finish the proof of Proposition~\ref{prop:graph} and hence of Theorem~\ref{thm:main}.

\begin{proof}[Proof of Proposition~\ref{prop:graph}]
Choose $N$ as in Lemma~\ref{lem:growth}, and consider the increasing
sequence of finite graphs $G_0\subset G_1\subset\cdots$ constructed above.
Let $G_\infty$ be their union, equipped with the metric $d_\infty$.

By Lemma~\ref{lem:growth}, $k_n=|V_n|>0$ for every $n$.
Proposition~\ref{prop:proper} gives properness, unboundedness, and girth
at least $2\pi$ for $G_\infty$.
Set $Y=G_\infty$, and let $A\subset G_0$ be the three-point set from
Proposition~\ref{prop:seed}.

With $\ell$ as in Proposition~\ref{prop:seed}, set
\begin{equation}\label{eq:beta-choice}
 \beta_1=\beta_2=\beta_3=\frac\ell2,\qquad
 \beta_{n+4}=\alpha_n\quad(n\geq0).
\end{equation}
All these numbers lie in $(0,\pi/2)$, and their squares have finite sum
by \eqref{eq:sums}.

Form the sets $A_i$ by \eqref{eq:controlled-hulls}, using the final metric.
The three steps in \eqref{eq:seed-generation} are also admissible in $Y$:
their geodesics have length at most $\ell<\pi$ and are preserved by
\eqref{eq:metric}. Hence $G_0\subset A_4$.
Every interval attached in the step $G_n\to G_{n+1}$ has length
$2\alpha_n=2\beta_{n+4}$ and endpoints in $V_n\subset G_n$.
By \eqref{eq:stage-hull}, induction gives
\[
 G_n\subset A_{n+4}\qquad(n\geq0).
\]
Taking the union proves \eqref{eq:controlled-exhaustion}.
\end{proof}

\end{document}